\documentclass[12pt]{article}
\author{Biao Wang}
\date{}

\usepackage{amsmath,amssymb,amsthm,amsfonts,mathrsfs,dsfont}
\usepackage{extarrows,chemarrow}
\usepackage{ulem}
\usepackage{accents}
\usepackage{graphicx}
\usepackage[all,cmtip]{xy}
\usepackage{amsrefs}

\usepackage[colorlinks=black,
linkcolor=black,
anchorcolor=black,
citecolor=black
]{hyperref}
\hypersetup{hidelinks}

\usepackage{fullpage}
\usepackage{setspace}
\newtheorem{thm}{Theorem}[section]
\newtheorem*{thm*}{Theorem}

\newtheorem*{cor*}{Corollary}
\newtheorem{lem}[thm]{Lemma}
\newtheorem*{lem*}{Lemma}

\newtheorem*{prop*}{Proposition}
\theoremstyle{definition}

\theoremstyle{remark}
\newtheorem{rmk}[thm]{Remark}

\newcommand{\N}{\mathbb{N}}

\renewcommand{\P}{\mathbb{P}}

\newcommand{\C}{\mathbb{C}}

\renewcommand{\leq}{\leqslant}
\renewcommand{\geq}{\geqslant}
\newcommand{\rr}{\rightarrow}

\newcommand{\of}[1]{\left(#1\right)}

\newcommand{\smat}[1]{\begin{smallmatrix}#1\end{smallmatrix}}

\begin{document}
\title{A generalization on the average ratio of the smallest and largest prime divisor of $n$}
\maketitle

\begin{abstract}
	In 1982, Erd$\ddot{\text{o}}$s and van Lint showed an estimate for the average of  the ratio of the smallest and largest prime divisor of $n$. In this note, we apply C.H. Jia's method to give an estimate for the average of positive integer power of the ratio.
\end{abstract}

\section{Introduction}
Let $n>1$ be an integer. Denote by $p(n)$ the smallest prime divisor of $n$ and $P(n)$ the largest prime divisor of $n$. Let $S(x)$ be the average of the ratio of the smallest and largest prime divisor of $n$:
$$S(x)=\sum_{n\leq x}\frac{p(n)}{P(n)}.$$ 
In 1982, Erd$\ddot{\text{o}}$s and van Lint \cite{erdos82} proved that
$$S(x)=\frac{x}{\log x}+\frac{3x}{\log^2 x}+o\of{\frac{x}{(\log x)^2}}.$$
In 1987, C.H. Jia \cite{jia87} proved that
$$S(x)=\frac{x}{\log x}+\frac{3x}{\log^2 x}+\frac{15x}{\log^3 x}+o\of{\frac{x}{(\log x)^3}}.$$

\quad In this note, we consider a generalization of the average $S(x)$ and find an estimate by applying C.H. Jia's method in \cite{jia87}. Let  $\omega(n)$ be the number of distinct prime divisors of $n$. Suppose $\lambda:\N \rr \C$ is a bounded arithmetic function. For a positive real number $\alpha>0$, let $S_{\lambda,\alpha}(x)$ be the weighted average of the $\alpha$-th power of the ratio of the smallest and largest prime divisor of $n$ with respect to $\lambda$ as follows
$$S_{\lambda,\alpha}(x)=\sum_{n\leq x}\lambda(\omega(n))\left(\frac{p(n)}{P(n)}\right)^\alpha.$$
Note that when $\lambda\equiv1$ and $\alpha=1$, $S_{\lambda,\alpha}(x)$ turns to be
$S(x)=\sum\limits_{n\leq x}\frac{p(n)}{P(n)}.$

\begin{thm}\label{mainthm1} 
	Let $S_{\lambda,\alpha}(x)$ be defined as above  and  $\pi(x)=\sum_{p\leq x}1$ be the prime counting function.  Then for $\alpha>\frac45$ we have
	\begin{equation}
		S_{\lambda,\alpha}(x)=(\lambda(1)+O(1))\pi(x)
	\end{equation}
and
\begin{equation}\label{maththmeq2}
\begin{split}
S_{\lambda,\alpha}(x)&=\lambda(1)\frac{x}{\log x}+\left(\frac{2}{{\alpha}}\lambda(2)+\lambda(1)\right)\frac{x}{\log^2 x}\\
&\quad+\left(\frac9{{\alpha}^2}\lambda(3)+\frac4{\alpha}\lambda(2)+2\lambda(1)\right)\frac{x}{\log^3 x}+O\of{\frac{x}{(\log x)^4}}.
\end{split}
\end{equation}
\end{thm}

\begin{rmk}
	Using the method of the proof, in principle one can find out the coefficient of the term $\frac{x}{\log^4x}$ and so on. But the computation is very complicated.
\end{rmk}

\section{Some Lemmas}
Before going to the proof of Theorem \ref{mainthm1}, we cite/improve some lemmas below.
	\begin{lem}[\cite{bom62}]\label{sum1lem}
		For any constant $A>0$, we have
		$$\pi(x)=\sum_{p\leq x}1=\int_2^x\frac{dt}{\log t}+O\of{\frac{x}{(\log x)^A}}.$$
	\end{lem}

\begin{lem}[Lemma 3, \cite{jia87}]\label{mainlem}
	Suppose $c$ is a constant, $f(x)$ is a function satisfying $f(x)=O\of{x^c}, f'(x)=O\of{x^{c-1}}$, $x\geq1$. Then for any constant $A>0$ and $\frac32<y\leq x$, we have
	\begin{equation}\label{maineq1}
	\sum_{y<p\leq x}f(p)=\int_y^x\frac{f(t)}{\log t}dt+O\of{\frac{x^{c+1}}{(\log x)^A}}, c>-1,
	\end{equation}
	\begin{equation}\label{maineq1-2}
	\sum_{y<p\leq x}f(p)=\int_y^x\frac{f(t)}{\log t}dt+O\of{\frac{y^{c+1}}{(\log y)^A}}, c<-1.
	\end{equation}
\end{lem}

\begin{lem}[Lemma 4, \cite{jia87}]\label{smoothnum}
	Let $\Psi(x,y)$ be the number of positive integers in $[1,x]$ whose prime factors are $\leq y$. If
	$$y=\exp\left(\frac{\log x}{\log\log x}\right),$$
	then for any $A>0$, we have
	$$\Psi(x,y)=O\left(\frac{x}{(\log x)^A}\right).$$
\end{lem}

\begin{lem}\label{mainlem2}
	For integer $i\geq 1$, let
	$\displaystyle{\Sigma^{(i)}=\sum_{ n\leq x,\mu(n)\neq0, \omega(n)=i}\left(\frac{p(n)}{P(n)}\right)^\alpha},$ where $\mu(n)$ is the M$\ddot{\text{o}}$bius function.
	Then for any constant $A>0$, we have
	\begin{equation}\label{maineq2}
	S_{\lambda,\alpha}(x)=\sum_{i=1}^\infty\lambda(i)\Sigma^{(i)} +O\of{\frac{x}{(\log x)^{ A}}}.
	\end{equation}
\end{lem}

\begin{proof} Since $\lambda$ is bounded, we have $\sum\limits_{ n\leq x,\mu(n)=0}\lambda(\omega(n))\left(\frac{p(n)}{P(n)}\right)^\alpha=O\left(\sum\limits_{ n\leq x,\mu(n)=0}\left(\frac{p(n)}{P(n)}\right)^\alpha\right)$. If $\alpha\geq1$, then $\left(\frac{p(n)}{P(n)}\right)^\alpha\leq \frac{p(n)}{P(n)}$. By the proof of Lemma 5 in  \cite{jia87}, we have $\sum\limits_{ n\leq x,\mu(n)=0}\frac{p(n)}{P(n)}=O\left(\frac{x}{(\log x)^{ A}}\right)$. So $\sum\limits_{ n\leq x,\mu(n)=0}\left(\frac{p(n)}{P(n)}\right)^\alpha=O\left(\frac{x}{(\log x)^{ A}}\right)$.

	 If $0<\alpha<1$, then $\left(\frac{p(n)}{P(n)}\right)^\alpha\leq 1$, we have 
	\begin{align*}
	\sum_{ n\leq x,\mu(n)=0}\left(\frac{p(n)}{P(n)}\right)^\alpha&=\sum_{ \smat{n\leq x,\mu(n)=0\\P(n)>p(n)(\log x)^{\frac{A}{\alpha}}}}\left(\frac{p(n)}{P(n)}\right)^\alpha+\sum_{ \smat{n\leq x,\mu(n)=0\\P(n)<p(n)(\log x)^{\frac{A}{\alpha}}}}\left(\frac{p(n)}{P(n)}\right)^\alpha\\
	&=O\left( \frac{x}{(\log x)^{ A}}\right)+O\Big(\sum_{ \smat{n\leq x,\mu(n)=0\\P(n)<p(n)(\log x)^{\frac{A}{\alpha}}}}1\,\,\Big)\\
	&=O\left( \frac{x}{(\log x)^{ A}}\right)+ O\left( \frac{x}{(\log x)^{ \frac{A}{\alpha}}}\right)=O\left( \frac{x}{(\log x)^{ A}}\right).
	\end{align*}
	Here the upper bound for the second $O$-term comes from the proof of Lemma 5 in  \cite{jia87}. 
	
	Therefore, for any constant $A>0$, 
	\begin{equation*}\label{maineq2-2}
	S_{\lambda,\alpha}(x)=\sum_{2\leq n\leq x,\mu(n)\neq0}\lambda(\omega(n))\left(\frac{p(n)}{P(n)}\right)^\alpha+O\of{\frac{x}{(\log x)^{A}}}
	\end{equation*}
and the lemma follows immediately.
\end{proof}

\quad By Lemma \ref{mainlem2},  to estimate $S_{\lambda,\alpha}(x)$, it suffices to estimate $\Sigma^{(i)}$ for each $i$. We shall use Lemma \ref{mainlem} repeatedly to get the estimates.

\section{Proof of Theorem \ref{mainthm1}}
	\subsection{Computation of \texorpdfstring{$\Sigma^{(1)}$}{}}
		\quad Clearly, $\Sigma^{(1)}=\pi(x)$. By Lemma \ref{sum1lem}, we get that
		\begin{equation}\label{sum1eq1}
		\Sigma^{(1)}=\pi(x)=\frac{x}{\log x}+\frac{x}{\log^2 x}+\frac{2x}{\log^3 x}+O\of{\frac{x}{\log^4 x}}.
		\end{equation}

	\subsection{Computation of \texorpdfstring{$\Sigma^{(2)}$}{}}
		\begin{align}
		\Sigma^{(2)}=\sum_{\smat{p_1p_2\leq x\\p_1<p_2}}\frac{p_1^{\alpha}}{p_2^{\alpha}}&=\sum_{p_2\leq\sqrt{x}}\frac1{p_2^{\alpha}}\sum_{p_1<p_2}p_1^{\alpha}+\sum_{\sqrt{x}<p_2\leq \sqrt{x}(\log x)^{\frac4\alpha} }\frac1{p_2^{\alpha}}\sum_{p_1\leq\frac{x}{p_2}}p_1^{\alpha}+O\of{\frac{x}{\log^4 x}} \nonumber\\
		&=I_1+I_2+O\of{\frac{x}{\log^4 x}}	
		\end{align}
\quad By Lemma \ref{mainlem}, we get
$$\sum_{p_1<p_2}p_1^{\alpha}=\int_{2}^{p_2}\frac{t^\alpha}{\log t}dt+O\of{\frac{p_2^{{\alpha}+1}}{(\log p_2)^3}}.$$
It follows that
		\begin{align}
		I_1&=\sum_{p_2\leq\sqrt{x}}\frac1{p_2^{\alpha}}\left(\frac{p_2^{{\alpha}+1}}{({\alpha}+1)\log p_2}+\frac{p_2^{{\alpha}+1}}{({\alpha}+1)^2(\log p_2)^2}+O\of{\frac{p_2^{{\alpha}+1}}{(\log p_2)^3}}\right)\nonumber\\
		&=\frac2{{\alpha}+1}\frac{x}{\log^2 x}+\frac{4{\alpha}+8}{({\alpha}+1)^2}\frac{x}{\log^3 x}+O\of{\frac{x}{\log^4 x}}
		\end{align}
\quad For $I_2$, notice that $\log \frac{x}{p_2}\asymp \log x$ for $\sqrt{x}<p_2\leq \sqrt{x}(\log x)^{\frac4\alpha}$. So by Lemma \ref{mainlem} again, 
		\begin{align}\label{i2}
		I_2
			&=\sum_{\sqrt{x}<p_2\leq \sqrt{x}(\log x)^{\frac4\alpha}}\frac1{p_2^{\alpha}}\left(\frac{\of{\frac{x}{p_2}}^{{\alpha}+1}}{({\alpha}+1)\log \frac{x}{p_2}}+\frac{\of{\frac{x}{p_2}}^{{\alpha}+1}}{({\alpha}+1)^2(\log\frac{x}{p_2})^2}\left(1+O\of{\frac1{\log x}}\right)\right)
	\end{align}	
By Lemma \ref{mainlem} and using substitution of variables, we can get 
	\begin{align*}
	\sum_{\sqrt{x}<p_2\leq \sqrt{x}(\log x)^{\frac4\alpha}}\frac1{p_2^{2\alpha+1}\log\frac{x}{p_2}}
	&=\int_{\sqrt{x}}^{\sqrt{x}(\log x)^{\frac4\alpha}}\frac{dt}{t^{2\alpha+1}\log\frac{x}{t}\log t}+O\of{\frac1{x^\alpha(\log x)^4}}\\
	&=\frac1{x^\alpha}\int_1^{(\log x)^{\frac4\alpha}}\frac{dt}{t^{2\alpha+1}\log\frac{\sqrt x}{t}\log (\sqrt{x}t)}+O\of{\frac1{x^\alpha(\log x)^4}}.
	\end{align*}
For the integral, we use integration by parts to get
	\begin{align*}
	&\quad\int_1^{(\log x)^{\frac4\alpha}}\frac{dt}{t^{2\alpha+1}\log\frac{\sqrt x}{t}\log (\sqrt{x}t)}\\
	&=\left[-\frac1{2\alpha}\frac1{t^{2\alpha}\log\frac{\sqrt x}{t}\log (\sqrt{x}t)}\right]_1^{(\log x)^{\frac4\alpha}}+\frac1{\alpha}\int_1^{(\log x)^{\frac4\alpha}}\frac{\log t}{t^{2\alpha+1}\log^2\frac{\sqrt x}{t}\log^2 (\sqrt{x}t)}dt\\
	&=\frac2\alpha\frac1{\log^2x}+O\of{\frac1{(\log x)^4}}.
	\end{align*}
	So
	\begin{equation}\label{i2-1}
\sum_{\sqrt{x}<p_2\leq \sqrt{x}(\log x)^{\frac4\alpha}}\frac1{p_2^{2\alpha+1}\log\frac{x}{p_2}}	=\frac2\alpha\frac1{x^\alpha\log^2x}+O\of{\frac1{x^\alpha(\log x)^4}}.
	\end{equation}
	Similarly,  we have
\begin{align}\label{i2-2}
\sum_{\sqrt{x}<p_2\leq \sqrt{x}(\log x)^{\frac4\alpha}}\frac1{p_2^{2\alpha+1}(\log\frac{x}{p_2})^2}&=\frac4\alpha\frac1{x^\alpha\log^3x}+O\of{\frac{1}{x^\alpha(\log x)^4}}
\end{align}
Plugging equations (\ref{i2-1}) and (\ref{i2-2}) into equation (\ref{i2}), we get
	\begin{equation}
	I_2=\frac2{{\alpha}({\alpha}+1)}\frac{x}{\log^2 x}+\frac{4}{{\alpha}({\alpha}+1)^2}\frac{x}{\log^3 x}+O\of{\frac{x}{(\log x)^4}}
	\end{equation}		
\quad Therefore, adding estimates for $I_1$ and $I_2$, we get
\begin{equation}
\Sigma^{(2)}=\frac2\alpha\frac{x}{\log^2 x}+\frac4\alpha\frac{x}{\log^3 x}+O\of{\frac{x}{\log^4 x}}.
\end{equation}

\subsection{Computation of \texorpdfstring{$\Sigma^{(3)}$}{}}
For $\Sigma^{(3)}$, similar to the $\Sigma^{(3)}$ in \cite{jia87}, we have	
	\begin{align}
		\Sigma^{(3)}&=\sum_{p_3\leq x^{\frac13}}\frac1{p_3^{\alpha}}\sum_{p_2<p_3}\sum_{p_1<p_2}p_1^{\alpha}+\sum_{x^{\frac13}<p_3\leq x^{\frac13}(\log x)^{\frac4\alpha}}\frac1{p_3^{\alpha}}\sum_{p_2\leq\sqrt{\frac{x}{p_3}}}\sum_{p_1<p_2}p_1^{\alpha}+ \nonumber\\
	&	\sum_{x^{\frac13}<p_3\leq x^{\frac13}(\log x)^{\frac4\alpha}}\frac1{p_3^{\alpha}}\sum_{\sqrt{\frac{x}{p_3}}<p_2<p_3}\sum_{p_1\leq\frac{x}{p_2p_3}}p_1^{\alpha}+O\of{\frac{x}{\log^4 x}} \nonumber\\
	&=I_3+I_4+I_5+O\of{\frac{x}{\log^4 x}}
	\end{align}
\quad For $I_3$, similar to the computation of $I_1$, we have
\begin{align}
I_3&=\sum_{p_3\leq x^{\frac13}}\frac1{p_3^{\alpha}}\sum_{p_2<p_3}\of{\frac1{{\alpha}+1}\frac{p_2^{{\alpha}+1}}{\log p_2}+O\of{\frac{p_2^{{\alpha}+1}}{\log^2 p_2}}} \nonumber\\
&=\sum_{p_3\leq x^{\frac13}}\frac1{p_3^{\alpha}}\of{\frac1{({\alpha}+1)({\alpha}+2)}\frac{p_3^{{\alpha}+2}}{\log^2 p_3}+O\of{\frac{p_3^{{\alpha}+2}}{\log^3 p_3}}} \nonumber\\
&=\frac9{({\alpha}+1)({\alpha}+2)}\frac{x}{\log^3 x}+O\of{\frac{x}{\log^4 x}}.
\end{align}
\quad For $I_4$, similar to the computation of $I_2$, we have
\begin{align}
I_4&=\sum_{x^{\frac13}<p_3\leq x^{\frac13}\log^4x}\frac1{p_3^{\alpha}}\left(\frac1{({\alpha}+1)({\alpha}+2)}\frac{\of{\sqrt{\frac{x}{p_3}}}^{{\alpha}+2}}{\log^2\sqrt{\frac{x}{p_3}}}\of{1+O\of{\frac1{\log x}}}\right) \nonumber\\
&=\frac{18}{{\alpha}({\alpha}+1)({\alpha}+2)}\frac{x}{\log^3 x}+O\of{\frac{x}{\log^4 x}}.
\end{align}
\quad For $I_5$, first we have
\begin{align}
I_5&=\sum_{x^{\frac13}<p_3\leq x^{\frac13}(\log x)^{\frac4\alpha}}\frac1{p_3^{\alpha}}\sum_{\sqrt{\frac{x}{p_3}}<p_2<p_3}\left( \frac1{{\alpha}+1}\frac{\of{\frac{x}{p_2p_3}}^{{\alpha}+1}}{\log \frac{x}{p_2p_3}}\of{1+O\of{\frac1{\log x}}}\right) \label{i5-1}
\end{align}
Then by Lemma \ref{mainlem}, 
\begin{align}
&\quad\sum_{\sqrt{\frac{x}{p_3}}<p_2<p_3}\frac1{p_2^{\alpha+1}\log \frac{x}{p_2p_3}}\nonumber\\
&=\int_{\sqrt{\frac{x}{p_3}}}^{p_3}\frac{dt}{t^{\alpha+1}\log \frac{x}{tp_3}\log t}+O\of{\frac1{(\sqrt{\frac{x}{p_3}})^\alpha(\log \sqrt{\frac{x}{p_3}})^4}}\nonumber\\
&=\left[-\frac1\alpha\frac1{t^{\alpha}\log \frac{x}{tp_3}\log t}\right]_{\sqrt{\frac{x}{p_3}}}^{p_3}+\frac1\alpha\int_{\sqrt{\frac{x}{p_3}}}^{p_3}\frac{\log \frac{t^2p_3}{x}}{t^{\alpha+1}\log^2 \frac{x}{tp_3}\log^2 t}dt+O\of{\frac1{(\sqrt{\frac{x}{p_3}})^\alpha(\log {\frac{x}{p_3}})^4}}\nonumber\\
&=\frac4\alpha\frac1{(\sqrt{\frac{x}{p_3}})^\alpha\log^2\frac{x}{p_3}}-\frac1\alpha\frac1{p_3^\alpha\log\frac{x}{p_2^2}\log p_3}+O\of{\frac1{(\sqrt{\frac{x}{p_3}})^\alpha(\log{\frac{x}{p_3}})^4}}\label{i5-2}
\end{align}
Notice that $\log{\frac{x}{p_3}}\asymp \log x$ for $x^{\frac13}<p_3\leq x^{\frac13}(\log x)^{\frac4\alpha}$. Plugging (\ref{i5-2}) into (\ref{i5-1}) we get
\begin{align}\label{i5}
I_5&=\frac{4x^{\frac{{\alpha}}2+1}}{{\alpha}({\alpha}+1)}\sum_{x^{\frac13}<p_3\leq x^{\frac13}(\log x)^{\frac4\alpha}}\frac1{p_3^{\frac{3{\alpha}}2+1}\log^2\frac{x}{p_3}}\of{1+O\of{\frac1{\log x}}} \nonumber\\
&\qquad-\frac{x^{{\alpha}+1}}{{\alpha}({\alpha}+1)}\sum_{x^{\frac13}<p_3\leq x^{\frac13}(\log x)^{\frac4\alpha}}\frac1{p_3^{3{\alpha}+1}\log\frac{x}{p_3^2}\log p_3} 
\end{align}
Now, similar to equation (\ref{i2-1}), we have
\begin{align}
\sum_{x^{\frac13}<p_3\leq x^{\frac13}(\log x)^{\frac4\alpha}}\frac1{p_3^{\frac{3{\alpha}}2+1}\log^2\frac{x}{p_3}}&=\int_{x^{\frac13}}^{x^{\frac13}(\log x)^{\frac4\alpha}}\frac{dt}{t^{\frac{3{\alpha}}2+1}\log^2\frac{x}t\log t}+O\of{\frac1{x^{\frac\alpha2}\log^4x}}\nonumber\\
&=\frac9{2\alpha}\frac{1}{x^{\frac\alpha2}\log^3x}+O\of{\frac1{x^{\frac\alpha2}\log^4x}}
\end{align}
and
\begin{align}
\sum_{x^{\frac13}<p_3\leq x^{\frac13}(\log x)^{\frac4\alpha}}\frac1{p_3^{3{\alpha}+1}\log\frac{x}{p_3^2}\log p_3}&=\int_{x^{\frac13}}^{x^{\frac13}(\log x)^{\frac4\alpha}}\frac{dt}{t^{3\alpha+1}\log\frac{x}{t^2}\log^2t}+O\of{\frac1{x^\alpha \log^4x}}\nonumber\\
&=\frac9\alpha\frac1{x^\alpha \log^3x}+O\of{\frac1{x^\alpha \log^4x}}.
\end{align}
Plugging  them into equation (\ref{i5}), we can get
\begin{equation}
I_5=\frac9{{\alpha}^2({\alpha}+1)}\frac{x}{\log^3 x}+O\of{\frac{x}{\log^4 x}}.
\end{equation}
\quad Therefore, 
\begin{equation}
\Sigma^{(3)}=\frac9{{\alpha}^2}\frac{x}{\log^3 x}+O\of{\frac{x}{\log^4 x}}.
\end{equation}

\subsection{Computation of the remainder}
Similar to the proof for the estimates of $\Sigma^{(4)}$ in \cite{jia87}, one can show that
\begin{align}
	\Sigma^{(4)}=O(\frac{x}{\log^4 x}),\\
	\Sigma^{(5)}=O(\frac{x}{\log^5 x}).
\end{align}
Notice that for $P(n)<p(n)\log^{\frac4\alpha} x$, we have
$$\left(\frac{p(n)}{P(n)}\right)^\alpha\leq (\log x)^{\frac4\alpha-4} \cdot \frac{p(n)}{P(n)}.$$
By Lemma \ref{sum1lem} and Lemma \ref{smoothnum}, we have
\begin{align*}
		\Sigma^{(6)}+	\Sigma^{(7)}+\cdots&=\sum_{2\leq n\leq x, \omega(n)\geq 6,\mu(n)\neq0}\left(\frac{p(n)}{P(n)}\right)^\alpha\\
		&=O\of{\sum_{\smat{2\leq n\leq x,
			P(n)>y\\
		P(n)<p(n)\log^{\frac4\alpha} x\\
	\omega(n)\geq 6,\mu(n)\neq0}}\left(\frac{p(n)}{P(n)}\right)^\alpha}+O\of{\frac{x}{\log^4x}}\\
&=O\of{(\log x)^{\frac4\alpha-4} \sum_{\smat{2\leq n\leq x,
			P(n)>y\\
			P(n)<p(n)\log^{\frac4\alpha} x\\
			\omega(n)\geq 6,\mu(n)\neq0}}\frac{p(n)}{P(n)}}+O\of{\frac{x}{\log^4x}}\\
		&=O\of{(\log x)^{\frac4\alpha-4}\sum_{p_6\log^{-\frac4\alpha} x<p_1<\cdots<p_6\leq x,\,p_6>y}\frac{x}{p_1\cdots p_6}}+O\of{\frac{x}{\log^4x}}\\
		&=O\of{x(\log x)^{\frac4\alpha-4}\sum_{y<p_6\leq x}\frac1{p_6}\left(\sum_{p_6\log^{-\frac4\alpha} x<p<p_6}\frac1p\right)^5}+O\of{\frac{x}{\log^4x}}\\
		&=O\of{x(\log x)^{\frac4\alpha-4}\sum_{y<p_6\leq x}\frac1{p_6}\left(\frac{\log\log x}{\log p_6}\right)^5}+O\of{\frac{x}{\log^4x}}\\
		&=O\of{\frac{x(\log\log x)^{11}}{(\log x)^{9-\frac\alpha4}}}+O\of{\frac{x}{\log^4x}}.
\end{align*}
If $\alpha>\frac45$, then 
\begin{equation}
	\Sigma^{(6)}+	\Sigma^{(7)}+\cdots=O\of{\frac{x}{\log^4x}}.
\end{equation}
\textit{Proof of Theorem\ref{mainthm1}}. Theorem \ref{mainthm1} follows by
combining the results in sections 3.1-3.4 and taking $A=4$ in Lemma \ref{mainlem2}.

\section{An application}
\quad Let $k\geq1$ be an integer. Taking $\lambda\equiv1$ and $\alpha=k$, we can get an estimation of $S_k(x)=\sum\limits_{n\leq x}\left(\frac{p(n)}{P(n)}\right)^k$ by Theorem \ref{mainthm1}. Suppose $f$ is a real analytic function with convergence radius greater than 1 and $f(0)=0$, then $f(x)=\sum\limits_{i=1}^\infty a_ix^i$. Let $S_f(x)=\sum\limits_{n\leq x}f\of{\frac{p(n)}{P(n)}}$. Then $S_f(x)=\sum\limits_{i=1}^\infty a_iS_i(x)$. Thus, by Theorem \ref{mainthm1} we can get another generalization of $S(x)$.

\begin{thm}\label{mainthm2}
	Suppose $f$ is a real analytic function with convergence radius greater than 1 and $f(0)=0$. Let $S_f(x)=\sum\limits_{n\leq x}f\of{\frac{p(n)}{P(n)}}$, then 
	\begin{equation*}\label{maththm2eq}
	\begin{split}
	S_f(x)&=f(1)\frac{x}{\log x}+\left(2\int_0^1\frac{f(t)}{t}dt+f(1)\right)\frac{x}{\log^2 x}\\
	&\qquad+\left(9\int_0^1\int_0^s\frac{f(t)}{st}dtds+4\int_0^1\frac{f(t)}{t}dt+2f(1)\right)\frac{x}{\log^3 x}+O\of{\frac{x}{(\log x)^3}}.
	\end{split}	
	\end{equation*}
\end{thm}

	\textbf{Acknowledgments}. The author would like to thank his advisor Prof. Xiaoqing Li for her useful suggestions and  would also like to thank Jiseong Kim for reading this note.

Department of Mathematics, University at Buffalo\\
\textit{Email}: \href{mailto: bwang32@buffalo.edu}{bwang32@buffalo.edu}\\

\end{document}